\documentclass[12pt,a4paper]{article}
\usepackage[a4paper]{geometry}
\usepackage{amsthm,amssymb}
\usepackage[alphabetic,abbrev]{amsrefs}
\usepackage{tikz}
\usepackage{breqn}
\usepackage{lmodern}
\usepackage{enumitem}
\usepackage{hyperref}
\usepackage[T1]{fontenc}
\usepackage[utf8]{inputenc}
\hypersetup{
   colorlinks,
   citecolor=red,
   filecolor=black,
   linkcolor=blue,
   urlcolor=green
}

\usepackage[titles]{tocloft}

\renewcommand{\geq}{\geqslant}

\usepackage{amsmath}
\usepackage{hyperref}
\usepackage{mathrsfs}
\usepackage[T1]{fontenc}
\usepackage[utf8]{inputenc}
\usepackage{lmodern}

\newtheorem{thm}{Theorem}

\newtheorem{lem}[thm]{Lemma}

\newtheorem{prop}[thm]{Proposition}

\theoremstyle{definition}
\newtheorem{rem}[thm]{Remark}
\begin{document}

\title{\bf The Modular Isomorphism Problem -- the alternative perspective on counterexamples \thanks{
This work was supported by the Grant WZ/WI-IIT/2/2022 from the Bialystok University of Technology
and funded from the resources for research by the Ministry of Science and Higher Education of Poland. }}
\author{Czes{\l}aw Bagi\'{n}ski \qquad Kamil Zabielski\\
\small Faculty of Computer Science\\[-0.8ex]
\small Bialystok University of Technology\\[-0.8ex]
\small Wiejska 45A Bialystok, 15-351 Poland\\
%%\small\tt \{c.baginski, kamil.zabielski\}@pb.edu.pl}
}
\date{}

\maketitle
\begin{abstract}
	As a result of impressive research \cite{gdelrma}, D. Garc\'{\i}a-Lucas, \'{A}. del R\'{i}o and L.~Margolis
    defined an infinite series of non-isomorphic $2$-groups $G$ and $H$, whose group algebras 
    $\mathbb{F}G$ and $\mathbb{F}H$ over the field $\mathbb{F}=\mathbb{F}_2$ are isomorphic, 
    solving negatively the long-standing Modular Isomorphism Problem (MIP).
	In this note we give a different perspective on their examples and show that they are special 
    cases of a more general construction. We also show that this type of construction for $p>2$ does 
    not provide a similar counterexample to the MIP.
\end{abstract}

Let $p$ be a prime number, $G$ and $H$ finite $p$-groups and $\mathbb{F}$ a field
of characteristic $p$. The long-standing conjecture, due to R. Brauer called the Modular Isomorphism
Problem (MIP), states that if $\mathbb{F}G$ and $\mathbb{F}H$ are isomorphic algebras then $G$ and $H$ 
are isomorphic groups. The problem is settled in the positive in many cases but recently for $p=2$ 
an infinite series of counterexamples was given (\cite{gdelrma}). In this paper we show that
these counterexamples are special cases of a more general construction for $p = 2$. We show also that 
for $p>2$ the analogous construction does not provide a similar counterexample to the MIP. Therefore the 
case $p>2$ remains open.\par 

\smallskip

For a survey of known results concerning the MIP see an excellent paper by L. Margolis \cite{margolis2022}.
The terminology used in the paper is standard. Our research was supported by extensive use of GAP 
software (\cite{GAP4}).\par

\smallskip

We begin with easy folklore observations.
Let $K$ be a dihedral group of order $2^{k+1}$,
\begin{equation}
    K = \langle t,\ s\, |\ t^2=s^{2}=(ts)^{2^k}=1\rangle.
\end{equation}

If we put $r = ts$, then this presentation can be replaced by the following one
\begin{equation}
    K = \langle t,\ r\, |\ t^2=r^{2^k}=1, r^t=r^{-1}\rangle.
\end{equation}

Let $C=\langle c\,|\ c^{2^n}=1\rangle$, $D=\langle d\,|\ d^{2^m}=1\rangle$, 
where $n > m \geqslant k \geqslant 3$. 
In the group $$P = K \times C \times D,$$ we define three elements $x=tc$, $y=sd$, $z = tsd = rd$ 
and two subgroups
\begin{equation*}
    \begin{aligned}
        & G = \langle x,y \rangle, & &  H = \langle x,z \rangle.
    \end{aligned}
\end{equation*}

\begin{lem}\label{lem:group-structure}
The following properties are satisfied:
\begin{enumerate}[noitemsep, label=\textup{\textrm{{(\roman*)}}}]
    \item $|G|=|H|=2^{n+m+k-1}$;\
    \item  $G'=H'=P'=K'=\langle (ts)^2 \rangle$ and the nilpotency class of groups $G,H,P,K$ is equal~to~$k$;\
    \item $\Phi(G) = \Phi(H) = \Phi(P)=\langle (ts)^2,\, c^2,\, d^2 \rangle$;\
    \item The subgroup $M=<ts,c,d>$ is abelian and maximal in $P$;\
    \item The subgroup $G \cap M = \langle xy, c^2, d^2 \rangle$ is maximal in $G$;\
    \item The subgroup $H \cap M = \langle z, c^2, d^2 \rangle$ is maximal in $H$;\
    \item The groups $G$ and $H$ are not isomorphic.
\end{enumerate}
\end{lem}

\begin{proof}
    (i) The projection of $P$ onto $K$ restricted to $G$ maps $G$ onto $K$. The kernel of this 
    restricted projection is equal to $G\cap (C\times D)=\langle c^2,\, d^2\rangle=\langle x^2,\, y^2\rangle$. 
    Therefore $|G|=|K|\cdot |\langle c^2,\, d^2\rangle|=2^{n+m+k-1}$. The same projection restricted to $H$ 
    is also a map of $H$ onto $K$ with the kernel $H\cap (C\times D)=\langle c^2,\, d^2\rangle=
    \langle x^2,\, (x^{-1}z)^2\rangle$, because $(x^{-1}z)^2=(c^{-1}t^{-1}tsd)^{2}=c^{-2}d^2$. Thus also
    $|H|=2^{n+m+k-1}$.\par\smallskip
    (ii) The commutator subgroup $K'$ of $K$ is cyclic and generated 
    by $${\left[ t,s \right] = t^{-1}s^{-1}ts=\left( ts \right)^{2} = r^{2} = \left[ t,r \right]}.$$ 
    But $[t,s]=[x,y]$ and $[t,r]=[x,z]$, so $K'=G'=H'$. It is also obvious by construction 
    of $P$, that $K'=P'$. It follows analogously that the next terms of the lower central series of these groups are equal.
    More precisely, $\gamma_{i} \left( X \right) = \left( ts \right)^{2^{i-1}}$ for $i \geq 2$ and $X \in \left\{ G, H, P, K \right\}.$ \par \smallskip
    (iii) It is clear that $\Phi \left( P \right) = 
    \Phi \left( K \right) \Phi \left(C \right) \Phi \left(D \right) = 
    \langle (ts)^2,\, c^2,\, d^2 \rangle$. Further $\left[ x,y \right]= \left[ tc,sd \right] 
    = \left[ t,s \right]= \left( ts \right)^2$, $x^{2}=\left(tc \right)^2=c^2$, $y^2=(sd)^2=d^2$, 
    hence~${\Phi \left(G \right) = \Phi\left( P \right)}$.
    As $\left[ x,z \right] = \left[ tc,tsd \right] = \left[ t,ts \right] = 
    t\left( st \right) t \left( ts \right) = \left( ts \right)^2$, 
                $x^{2} = c^{2}$, 
                $z^{2} = \left( tsc\right)^{2} = \left(ts \right)^{2} c^{2}$, we~have $\Phi \left( H \right)=
                \Phi \left( P \right)$.\par \smallskip
    (iv-vi) $M$ is abelian, because $c$ and $d$ are central in $P$. Moreover 
    $|M| = o \left(ts \right) \cdot o\left(c \right) \cdot o \left(d \right) = 2^{k+m+n}$, 
    which means that $|P:M|=2$. The maximality of $G\cap M$ in $G$ and $H\cap M$ in $H$ is obvious as $M$ is~abelian. \par \smallskip
    (vii) By the assumption $n > m \geq k$, the order of $xy$ is equal to $2^n$,  
    as~$(xy)^2=(tcsd)^2=(ts)^2c^2d^2$. This is maximal order of generators of $G\cap M$,  therefore
    $\exp \left(G \cap M \right)=2^{n}$. 
    Similarly, the element $z$ has order $2^{m}$, so $c^2$ has maximal order among generators of $H\cap M$ 
    and then $\exp \left( H \cap M \right) = 2^{n-1}$.
    It is well-known that if a non-abelian $p$-group $X$ has two different abelian subgroups of index $p$, then $\Phi(X)$ is in the center of $X$, in particular $X$ is~of~class of~$2$.
    Since $G$ and $H$ have nilpotency class $k \geq 3$ the subgroups $G \cap M$ and $H \cap M$ are the unique abelian subgroups of index $2$ in $G$ and $H$ respectively.
    Now groups $G\cap M$ and $H\cap M$ have different exponents, therefore they are not isomorphic. Hence $G$ and $H$ are not isomorphic as well.
\end{proof}

\begin{lem}\label{lem:group-relations}
    The groups $G$ and $H$ can be described in terms of generators and relations in the following way:
    \begin{equation}\label{relations}
        \begin{aligned}
            G & = \langle x,\, y,\ u\, \mid x^{2^n} = y^{2^m} = u^{2^{k-1}} = 1,\ y^{x}=yu,\ u^{x}=u^{-1},\ u^{y} = u^{-1} \rangle; \\
            H & = \langle x,\, z,\ u\, \mid x^{2^n} = z^{2^m} = u^{2^{k-1}} = 1,\ z^{x}=zu,\ u^{x}=u^{-1},\ u^{z}=u        \rangle.
        \end{aligned}
    \end{equation}
\end{lem}

\begin{proof}
    The correspondence $x\rightarrow tc$, $y\rightarrow sd$ can be extended to a homomorphism of $G$ described 
    by (\ref{relations}) into $P$ because $tc$ and $sd$ satisfy the relations (\ref{relations}) for $G$. Since 
    $\langle tc,\, sd\rangle$ generate a subgroup of order $2^{k+m+n-1}$ this homomorphism is an embedding.
    One can use a similar argument for the group $H$.
\end{proof}

\begin{lem}\label{lem:group-relations-reverse}
    Let $A$ be a $2$-group generated by elements $a$ and $b$ of order $2^n$ and $2^m$ respectively such
    that $a^2$ and $b^2$ are central in $A$, $|A'|=2^{k-1}$ and $\langle a^2,\, b^2\rangle\cap A'=1$. 
    Then $A$ is isomorphic to $G$. 
\end{lem}
\begin{proof}
    Note that  $A / \langle a^{2}, b^{2} \rangle$ as a group generated by two elements of order $2$ is dihedral of~order~$2^{k+1}$.
    Since $x^{2}=c^{2}, y^{2}=d^{2}$ are central in $G$, $| G' | = 2^{k-1}$ and $\langle x^{2}, y^{2} \rangle \cap G' = 1$.
    Therefore the correspondence $a \rightarrow x$, $b \rightarrow y$ extends in a natural way to an isomorphism of~$A$~onto~$G$.
\end{proof}

\begin{lem}\label{lem:group-in-group-algebra}
    In the group algebra $\mathbb{F}H$ let $\beta = 1 + x \left( 1 + z \right)$. 
    Then 
    \begin{enumerate}[noitemsep, label=\textup{\textrm{{(\roman*)}}}]
        \item $\beta $ has order $2^k$ in the group of units of $\mathbb{F}H$ and  $\beta^2$ is a central element of $\mathbb{F}H$;
        \item the subgroup $\langle x,\,\beta\rangle$ is isomorphic to $G$ and spans $\mathbb{F}H$.
    \end{enumerate} 
\end{lem}

\begin{proof}
    \begin{enumerate}[noitemsep, label=\textup{\textrm{{(\roman*)}}}]\item Note first that 
    \begin{equation*}
        \begin{aligned}
            \beta^2 &= \left( 1+x \left( 1+z \right) \right)^2 
                     = 1+x \left( 1+z \right) x \left( 1+z \right) 
                     = 1 + x^{2} \left( 1 + z^{x} \right) \left( 1+z \right) = \\
                    &= 1 + x^{2} \left( 1 + \left( z + z^x \right) + \left( z^xz \right) \right).
        \end{aligned}
    \end{equation*}
Further
$z+z^x=tsd+tsd^{tc}=tsd+ts^td=td(s+s^t)=tsd(1+(st)^2)=z(1+[z,t])$ and
$z^xz=(tsd)^{tc}(tsd)=(ts^td)(tsd)=d^2=z^2$, so

\begin{equation*}
    \beta^{2^{m-1}} = 1+x^{2^{m-1}}(1+z^{2^{m-2}}(1+(st)^{2^{m-1}})+z^{2^{m-1}}) \neq 1 \\ 
\end{equation*}
and
\begin{equation*}
    \beta^{2^m}     = 1+x^{2^{m}}(1+z^{2^{m-1}} (1 + (st)^{2^{m}}) + z^{2^{m}}) = 1,
\end{equation*}

because ${\left(st\right)^{2^{m}}} = z^{2^{m}} = 1.$
We have also $(\beta^2)^x=1+x^2(1+(z^x+z)+zz^x)=\beta^2$, because $z^xz=zz^x=d^2$. 

\item Consider the group $\tilde{G}=\langle x, \beta\rangle$. 
      Since $x^2, \beta^2\in \mathcal{Z}(\mathbb{F}H)$ the factor group $\tilde{G}/\langle x^2,\beta^2\rangle$ is a dihedral group. 
      The commutator subgroup of $\tilde{G}$, generated by the element
    \begin{equation*}
        \begin{aligned}
            \left[\beta,x\right] & = 1 + \beta^{-1}x^{-1}(\beta x+x\beta)
                                = 1 + \beta^{-1}(\beta^x+\beta)                   \\
                                & = 1 + \beta^{-1}x(z^x+z) 
                                = 1 + \beta^{-1}xz([z,x]+1)
                                = 1 + \beta^{-1}xz((ts)^2+1),
        \end{aligned}
    \end{equation*}
has order $2^{k-1}$ and has trivial intersection with $\langle x^2,\, \beta^2\rangle$.
\end{enumerate} 
\noindent It follows from Lemma \ref{lem:group-relations-reverse} that $\tilde{G}$ is isomorphic to $G$.

At the very end, we observe that $x$ and $\beta = 1+x \left(1+z \right)= z+\left(x+1\right)\left(z+1\right)$ 
are linearly independent modulo $A^2(\mathbb{F}H)$, where $A(\mathbb{F}H)$ is the augmentation ideal of 
$\mathbb{F}H$. 
Therefore by Jennings theory~\cite[Chapter 3.3]{passman1977} both generate the whole algebra $\mathbb{F}H$.
\end{proof}

As an immediate consequence of the lemmas we obtain the following 

\begin{thm}
For every $m,\ n,\ k$, $n>m\geqslant k\geqslant 3$ there exist non-isomorphic $2$-groups $G$~and~$H$ of 
order $2^{m+n+k-1}$ and cyclic commutator subgroup of order $2^{k-1}$ whose group algebras
over the $2$-element field are isomorphic.
\end{thm}

\begin{rem}\label{rem:last-remarkable}
    \begin{enumerate}[noitemsep, label=\textup{\textrm{{(\roman*)}}}]
        \item There are many elements in $\mathbb{F}H$ that, together with the element $x$, generate a group isomorphic to $G$ spanning $\mathbb{F}H$. 
              For instance, if $\zeta$ is a central unit of $\mathbb{F}H$ of order $<2^m$, $\tilde{x}$ is an element of $H$ such that $z^{\tilde{x}} \neq z$, then 
              $\tilde{\beta}=\zeta+\tilde{x}(1+z)$ is such an element.
        \item If we put $k=3$, we get the examples from~\cite{gdelrma} with $\beta=d^2+zx[z,x](1+z)$.
        \item If, in the above construction, we replace the dihedral group $K$ with the semidihedral or the generalized quaternion group of order $2^{k+1}$, we get the same groups (up~to~isomorphism) as the~groups $G$ and~$H$.
    \end{enumerate}
\end{rem}

It follows from Lemma 1, that in the counterexample the abelian subgroups of index $2$ in 
$G$ and $H$ are not isomorphic. Let us consider analogous construction as in the counterexample taking $p>2$. 
As a basis of the construction take a $p$-group $K =\langle s, s_1\rangle$ of maximal class
 (similarly as in the case $p = 2$) having an abelian subgroup of index $p$ (all they are described; see for~instance~\cite[Section 4 and Theorem 4.3]{blackburn}).
 Then in the group $K \times \langle c \mid c^{p^{m}} = 1 \rangle \times \langle d\mid d^{p^{n}} = 1\rangle$, where $m,n \in \mathbb{N}$ are arbitrarily fixed, 
 take a subgroup~$G=\langle sc, s_1d\rangle$.
 In this group, the subgroup $N = C_{G} \left( G' / \Phi \left( G' \right)\right)$ is abelian and has index $p$ in $G$.

 The following lemma is a~consequence of Proposition~1.4~\cite{bagincaranti} and Lemma~1~\cite{bbaginski1988}.
 \begin{lem}
     \label{lem:baginski-caranti}
     If the subgroup $N$ of $G$ defined by equation $N = C_{G} \left( G' / \Phi \left( G' \right) \right)$ has index $p$ in~$G$, then the subgring $I \left( N \right) + I \left( G'\right)\mathbb{F}G$ is determined by the structure of $\mathbb{F}G$. 
     In particular the ideal $I \left(N \right) \mathbb{F}G$ of $\mathbb{F}G$ and the orders of factors $\mathcal{M}_{i}(N) / \mathcal{M}_{i+1}(N)$ of the Brauer-Jennings-Zassenhaus series of~$N$ are~determined.
 \end{lem}

\begin{prop}\label{prop}
    Let $G$ be a finite $2$-generated $p$-group in which the subgroup
$$N = C_{G} \left( G' / \Phi \left( G' \right) \right)$$ is abelian and has index $p$ then 
$N$ is determined by $\mathbb{F}G$.
\end{prop}

\begin{proof}
Let $H$ be a $p$-group such that $\mathbb{F}G\cong\mathbb{F}H$, then $H$ is~also $2$-generated and contains 
a subgroup $M = C_{H} \left( H' / \Phi\left( H'\right)\right)$ of index $p$.  
Suppose first that $M$ is abelian. It is well-known, see the proof~of~\cite[Lemma 14.2.7]{passman1977}, that
the series $\left\{ |\mathcal{M}_{i} (M) /\mathcal{M}_{i+1} (M)|, i \geq 1 \right\}$ determines 
the ismorphism class of $M$. Since 
$|\mathcal{M}_{i} (M) /\mathcal{M}_{i+1} (M)| = |\mathcal{M}_{i} (N) /\mathcal{M}_{i+1} (N)|,\ i\geq 1,$
 by~Lemma~\ref{lem:baginski-caranti}, we obtain $M \cong N$.
Therefore to complete the proof we need to show that $M$ is abelian.

It is true, if we assume additionally that $G$ is a $p$-group of maximal class, by 
\cite[Theorem 3.2]{bagincaranti}. We employ the proof of this theorem in the context of our case.

The set $\Phi(N)-Z(N)$ consists of all elements $g\in G$ that are the $p$-th powers of the non-central elements of $G$. 
Since $N$ is abelian and $|G:N|=p$, the $G$-class sums $\widehat{C_g}$ of these elements are the 
only $p$-th powers of other $G$-class sums in $\mathbb{F}G.$ It follows from \cite{parmenter1981} 
that the number of such class sums is determined by $\mathbb{F}G$, so the number of $H$-class sums 
$\widehat{C_h}$, $h\in H$, which are $p$-th powers is the same in $\mathbb{F}H$. Since 
$N/\Phi(G')\cong M/\Phi(H')$ this property have exactly the class sums $\widehat{C_h}$, if 
$h\in \Phi(M)-Z(H)$. The number of $G$-class sums which are $p$-th powers is as large as possible, 
because all they have cardinality $p$. Therefore all $H$-class sums of noncentral elements 
of $\Phi(M)$ have cardinality $p$, because $|\Phi(N)-Z(G)|=|\Phi(M)-Z(H)|$.

Let $H=\langle h,y\rangle$, where $y\in M$. If $\Phi(H)=1$, then $M$ is abelian, so assume that 
$\Phi(M)\neq 1$. Now it is clear that $y^p\notin \Phi(H')$, which implies that $y^p$ is not central, 
because $|H'\cap Z(H)|=p$ and then $H'\cap Z(H)\leqslant\Phi(H')$. Thus $\widehat{C_{y^p}}=
\widehat{C_y}^p$ and then $H:C_H(y)|=|H:C_H(y^p)|=p$, by \cite{sehgal1967}. Therefore $C_H(y)=M$, 
or equivalently $y\in Z(M)$. Now, $M$ is generated by $Z(H)$ and all conjugates of $y$, which are 
also contained in $Z(M)$. Hence $M$ is abelian. This ends the proof. 

\end{proof}

In the counterexample the subgroup $C_{G} \left( G' / \Phi\left( G'\right)\right) = G$ hence Proposition~\ref{prop} can not be applied for $p=2$.
Simultanously, Proposition~\ref{prop} shows that there are no similiar counterexamples for odd primes.

\end{document}